\documentclass[]{interact}

\usepackage{epstopdf}% To incorporate .eps illustrations using PDFLaTeX, etc.
\usepackage[caption=false]{subfig}% Support for small, `sub' figures and tables
\usepackage{graphicx}
\usepackage{amsfonts}
\usepackage{amssymb}
\usepackage{amsmath}
\usepackage{enumerate}
\usepackage{latexsym}
\usepackage[numbers,sort&compress]{natbib}% Citation support using natbib.sty
\bibpunct[, ]{[}{]}{,}{n}{,}{,}% Citation support using natbib.sty
\makeatletter% @ becomes a letter
\def\NAT@def@citea{\def\@citea{\NAT@separator}}% Suppress spaces between citations using natbib.sty
\makeatother% @ becomes a symbol again

\theoremstyle{plain}% Theorem-like structures provided by amsthm.sty
\newtheorem{theorem}{Theorem}[section]
\newtheorem{lemma}[theorem]{Lemma}

\theoremstyle{definition}

\theoremstyle{remark}
\newtheorem{remark}{Remark}

\begin{document}

%\articletype{ARTICLE TEMPLATE}% Specify the article type or omit as appropriate

\title{Closed Range Integral Operators  on Hardy, BMOA and Besov Spaces}

\author{
\name{Kostas Panteris\thanks{Contact Kostas Panteris  Email: mathp289@math.uoc.gr, kpanteris@yahoo.gr}}
\affil{Department of Mathematics and Applied Mathematics, University of Crete, University Campus Voutes, 70013 Heraklion, Greece }
}

\maketitle

\begin{abstract}
If $g\in H^{\infty}$, the integral operator $S_{g}$ on $H^{p}$, $BMOA$ and $B^{p}$(Besov) spaces, is defined as $S_{g}f(z)=\int_{0}^{z} f^{\prime}(w) g(w) dw$.  In this paper, we prove three necessary and sufficient conditions for the operator $S_{g}$ to have closed range on $H^{p}\hspace{1mm}(1 < p < \infty)$, $BMOA$ and $B^{p}\hspace{1mm}(1 < p < \infty)$.
\end{abstract}

%\begin{keywords}
%Composition operators; Closed Range; Hardy spaces; Reverse Carleson Measures
%\end{keywords}

\section{Introduction and Preliminaries}
 Let $\mathbb{D}$ denote the open unit disk in the complex plane, $\mathbb{T}$ the unit circle, $A$ the normalized area Lebesgue measure in $\mathbb{D}$ and $m$ the normalized length Lebesgue measure in $\mathbb{T}$. For $1 \leq p < \infty$ the Hardy space $H^{p}$ is defined as the set of all analytic functions $f$ in $\mathbb{D}$ for which
\[
\sup \limits_{0\leq r<1} \int \limits_{\mathbb{T}}  \vert f(r \zeta) \vert^{p} dm(\zeta) < +\infty
\] 
and the corresponding norm in $H^{p}$ is defined by
\[
\Vert f  \Vert_{H^{p}}^{p}  = \sup \limits_{0 \leq r<1} \int \limits_{\mathbb{T}}  \vert f(r \zeta) \vert^{p} dm(\zeta).
\]
When $p=\infty$, we define $H^{\infty}$ to be the space of bounded analytic functions $f$ in $\mathbb{D}$ and $\Vert f \Vert_{\infty} = \sup\lbrace \vert f(z)\vert:z\in\mathbb{D}\rbrace$.

In this work we will mainly make use of the following equivalent norm (see Calderon's theorem in \cite{Pavlovic}, page 213) in $H^{p}$, $1 \leq p < \infty$:
\begin{equation}\label{Stolz_norm}
\Vert f  \Vert_{H^{p}}^{p}  = \vert f(0)  \vert^{p} + \int \limits_{\mathbb{T}} \Big(\iint \limits_{\Gamma_{\beta}(\zeta)}  \vert f^{\prime}(z)\vert^{2} dA(z) \Big)^{\frac{p}{2}} dm(\zeta),
\end{equation}
where $\Gamma_{\beta}(\zeta)$ is the Stolz angle at $\zeta\in\mathbb{T}$, the conelike region with aperture $\beta\in(0,1)$, which is defined as
\[
\Gamma_{\beta}(\zeta) = \lbrace z\in\mathbb{D}: \vert z\vert < \beta\rbrace \cup \bigcup_{\vert z\vert< \beta}[z,\zeta).
\]
The BMOA space is defined as the set of all analytic functions $f$ in $\mathbb{D}$ for which
\[
\sup \limits_{\beta\in\mathbb{D}} \iint \limits_{\mathbb{D}} \frac{1-\vert\beta\vert^{2}}{\vert 1-\overline{\beta}z\vert^{2}} \vert f^{\prime}(z) \vert^{2} \log\frac{1}{\vert  z \vert} dA(z) <\infty
\]
and we may define the corresponding norm in BMOA by
\[
\Vert f  \Vert_{*}^{2}  = \vert f(0)  \vert^{2} + \sup \limits_{\beta\in\mathbb{D}} \iint \limits_{\mathbb{D}} \frac{1-\vert\beta\vert^{2}}{\vert 1-\overline{\beta}z\vert^{2}} \vert f^{\prime}(z) \vert^{2} \log\frac{1}{\vert  z \vert} dA(z).
\]
For $1 < p < \infty$ the Besov space $B^{p}$ is defined as the set of all analytic functions $f$ in $\mathbb{D}$ for which
\[
\iint \limits_{\mathbb{D}} \vert f^{\prime}(z) \vert^{p} (1-\vert z \vert^{2})^{p-2} dA(z) < +\infty 
\]
and the corresponding norm in $B^{p}$ is defined by
\[
\Vert f \Vert_{B^{p}}^{p} = \vert f(0) \vert^{p} + \iint \limits_{\mathbb{D}} \vert f^{\prime}(z) \vert^{p} (1-\vert z \vert^{2})^{p-2} dA(z).
\]
Let $g:\mathbb{D} \rightarrow \mathbb{C}$ be an analytic function. If $X$ is a space of analytic functions $f$ in $\mathbb{D}$ (in particular, in this paper, $X=H^{p}$ or $X=BMOA$ or $X=B^{p}$) then, the integral operator $S_{g}:X\rightarrow X$, induced by $g$, is defined as 
\[
S_{g}f(z) = \int_{0}^{z} f^{\prime}(w) g(w) dw, \hspace{2mm} z\in\mathbb{D},
\]
for every $f\in X$.

Let $\rho(z,w)$ denote the pseudo-hyberbolic distance between $z,w \in \mathbb{D}$,
\[
\rho(z,w) = \Big\vert \frac{z - w}{1 - \overline{z}w} \Big\vert,
\] 
 $D_{\eta}(a)$ denote the pseudo-hyberbolic disk of center $a \in \mathbb{D}$ and radius $\eta<1$:
\[
D_{\eta}(a) = \lbrace z \in \mathbb{D}: \rho(a,z) < \eta  \rbrace,
\]
and $\Delta_{\eta}(\alpha)$ denote the euclidean disk of center $a \in \mathbb{D}$ and radius $\eta(1-\vert \alpha\vert)$, $\eta<1$:
 \[
 \Delta_{\eta}(\alpha) = \lbrace z\in\mathbb{D}: \vert z - \alpha\vert< \eta(1-\vert \alpha\vert)  \rbrace.
 \]

In the following, $C$ denotes a positive and finite constant which may change from one occurrence to another. Moreover, by writing
$K(z) \asymp L(z)$ for the non-negative quantities $K(z)$ and $L(z)$ we mean that $K(z)$ is comparable to $L(z)$ if $z$ belongs to a specific set: there are positive constants 
 $C_{1}$ and $C_{2}$ independent of $z$ such that
\[
C_{1} K(z) \leq L(z) \leq C_{2} K(z).
\]

\section{Closed range integral operators on Hardy spaces}
Let $g:\mathbb{D} \rightarrow \mathbb{C}$ be an analytic function and, for $c>0$, let $G_{c} = \lbrace z\in\mathbb{D}:\vert g(z)\vert>c \rbrace$. It is well known (see \cite{Anderson}) that the integral operator $S_{g}:H^{p}\rightarrow H^{p}$ $(1\leq p < \infty)$ is bounded if and only if $g\in H^{\infty}$. 

We say that $S_{g}$, on $H^{p}$, is bounded below, if there is $C>0$ such that $\Vert S_{g}f\Vert_{H^{p}} > C \Vert f\Vert_{H^{p}}$ for every $f\in H^{p}$. Since $S_{g}$ maps every constant function to the $0$ function, if we want to study the property of being bounded below for $S_{g}$, we are obliged to consider spaces of analytic functions modulo the constants or, equivalently, spaces of analytic functions $f$ such that $f(0)=0$. Theorem 2.3 in \cite{Anderson} states that $S_{g}$ is bounded below on $H^{p}/\mathbb{C}$ if and only if  it has closed range on $H^{p}/\mathbb{C}$. Next, we denote $H^{p}/\mathbb{C}$ as $H_{0}^{p}$.

Corollary 3.6 n \cite{Anderson} states that
$S_{g}:H_{0}^{2}\rightarrow H_{0}^{2}$ has closed range if and only if
there exist $c>0$, $\delta > 0$ and $\eta \in (0,1)$ such that
\[
A(G_{c} \cap D_{\eta}(a)) \geq \delta A(D_{\eta}(a))
\]
for all $a \in \mathbb{D}$.

In the end of \cite{Anderson} A. Anderson posed the question, if the above condition for $H_{0}^{2}$  holds also for all $H_{0}^{p}$. In this paper, theorem \ref{integral_theorem} gives an affirmative answer to this question, for the case $1 < p < \infty$. Although the answer in case $p=2$ is an immediate consequence of D. Luecking's theorem (see \cite{Anderson}, Proposition 3.5), the answer in case $1 < p < \infty$ requires much more effort.

For $\lambda\in(0,1)$ and $f\in H^{p}$ we set 
$$
E_{\lambda}(\alpha) = \lbrace z\in \Delta_{\eta}(\alpha): \vert f^{\prime}(z) \vert^{2} > \lambda \vert f^{\prime}(\alpha)\vert^{2}  \rbrace
$$ 
and
$$
B_{\lambda}f(\alpha) = \frac{1}{A(E_{\lambda}(\alpha))} \iint \limits_{E_{\lambda}(\alpha)} \vert f^{\prime}(z)\vert^{2} dA(z).
$$
Lemma \ref{Luecking_lemma1} is due to D. Luecking (see \cite{Luecking81}, lemma 1).
\begin{lemma}\label{Luecking_lemma1}
Let $f$ analytic in $\mathbb{D}$, $a\in\mathbb{D}$ and $\lambda\in(0,1)$. Then
\[
\frac{A(E_{\lambda}(\alpha))}{A(\Delta_{\eta}(\alpha))} \geq \frac{\log\frac{1}{\lambda}}{\log\frac{B_{\lambda}f(\alpha)}{\vert f^{\prime}(\alpha) \vert^{2}} + \log\frac{1}{\lambda}}.
\]
\end{lemma}
Moreover in \cite{Luecking81}, the following sentence is proved: If $\alpha\in\mathbb{D}$ and $\frac{2\eta}{1+\eta^{2}}\leq r<1$ then
\begin{equation}\label{euclidean_to_pseudo}
\Delta_{\eta}(\alpha) \subseteq D_{r}(\alpha).
\end{equation}
We proceed with the main result of this section.
\begin{theorem}\label{integral_theorem}
Let $1 < p < \infty$ and $g\in H^{\infty}$. Then the following are equivalent:
\begin{enumerate}
\item[(i)] $S_{g}:H_{0}^{p}\rightarrow H_{0}^{p}$ has closed range 
\item[(ii)] There exist $c>0$, $\delta > 0$ and $\eta \in (0,1)$ such that
\begin{equation}\label{second_part}
A(G_{c} \cap D_{\eta}(a)) \geq \delta A(D_{\eta}(a))
\end{equation}
for all $a \in \mathbb{D}$.
\item[(iii)] There exist $c>0$, $\delta > 0$ and $\eta \in (0,1)$ such that
\begin{equation}\label{second_part2}
A(G_{c} \cap \Delta_{\eta}(a)) \geq \delta A(\Delta_{\eta}(a))
\end{equation}
for all $a \in \mathbb{D}$.
\end{enumerate}
\end{theorem}

We first prove two lemmas which will play an important role in the proof of theorem \ref{integral_theorem}. 

For $\zeta\in\mathbb{T}$ and $0<\beta<\beta^{\prime}<1$ we consider the Stolz angles $\Gamma_{\beta}(\zeta)$ and $\Gamma_{\beta^{\prime}}(\zeta)$, where $\beta^{\prime}$ has been chosen so that $\Delta_{\eta}(\alpha) \subset \Gamma_{\beta^{\prime}}(\zeta)$ for every $\alpha\in\Gamma_{\beta}(\zeta)$.

 \begin{lemma}\label{lemma_2}
  Let $\varepsilon>0$, $f$ analytic in $\mathbb{D}$ and
 \[
 A = \Big\lbrace \alpha\in\mathbb{D}: \vert f^{\prime}(\alpha)\vert^{2} < \frac{\varepsilon}{A(\Delta_{\eta}(\alpha))} \iint \limits_{\Delta_{\eta}(\alpha)} \vert f^{\prime}(z)\vert^{2} dA(z) \Big\rbrace.
 \]
 There is $C>0$ depending only on $\eta$ such that
 \[
 \iint \limits_{A\cap\Gamma_{\beta}(\zeta)} \vert f^{\prime}(z)\vert^{2} dA(z) \leq \varepsilon C \iint \limits_{\Gamma_{\beta^{\prime}}(\zeta)} \vert f^{\prime}(z)\vert^{2} dA(z)
 \]
 \end{lemma}
 \begin{proof}
 Integrating
 \[
  \vert f^{\prime}(\alpha)\vert^{2} < \frac{\varepsilon}{A(\Delta_{\eta}(\alpha))} \iint \limits_{\Delta_{\eta}(\alpha)} \vert f^{\prime}(z)\vert^{2} dA(z)
 \]
over $\alpha\in A\cap\Gamma_{\beta}(\zeta)$ and using Fubini's theorem on the right side, we get
 \[
 \iint \limits_{A\cap\Gamma_{\beta}(\zeta)} \vert f^{\prime}(\alpha)\vert^{2} dA(\alpha) < \varepsilon \iint \limits_{\Gamma_{\beta^{\prime}}(\zeta)} \vert f^{\prime}(z)\vert^{2} \Big[\iint \limits_{A\cap\Gamma_{\beta}(\zeta)} \frac{\chi_{\Delta_{\eta}(\alpha)}(z)}{A(\Delta_{\eta}(\alpha))} dA(\alpha)\Big] dA(z)
 \]
 Using \eqref{euclidean_to_pseudo} with $r=\frac{2\eta}{1+\eta^{2}}$, we have $\chi_{\Delta_{\eta}(\alpha)}(z) \leq \chi_{D_{r}(\alpha)}(z)=\chi_{D_{r}(z)}(\alpha)$. We have that $A(D_{r}(z))\asymp (1-\vert z\vert)^{2}$ and, for $\alpha\in D_{\eta}(z)$, we have $(1-\vert z\vert) \asymp (1-\vert \alpha\vert)$, where the underlying constants in these relations depend only on $\eta$. In addition, $A(\Delta_{\eta}(\alpha)) = \eta^{2}(1-\vert \alpha\vert)^{2}$.
 So, 
 \begin{align}\label{brackets_int}
 \iint \limits_{A\cap\Gamma_{\beta}(\zeta)} \frac{\chi_{\Delta_{\eta}(\alpha)}(z)}{A(\Delta_{\eta}(\alpha))}    dA(\alpha) & \leq \iint \limits_{A\cap\Gamma_{\beta}(\zeta)} \frac{\chi_{D_{r}(z)}(\alpha)}{\eta^{2}(1-\vert \alpha\vert)^{2}} dA(\alpha)\nonumber\\ 
 & \leq C \iint \limits_{D_{r}(z)} \frac{1}{\eta^{2}(1-\vert z\vert)^{2}} dA(\alpha) = C \frac{A(D_{r}(z))}{\eta^{2}(1-\vert z\vert)^{2}} \leq C,
\end{align}
where $C>0$ depends only on $\eta$.
 \end{proof}
 
 \begin{lemma}\label{lemma_3}
 Let $0<\varepsilon<1$, $f$ analytic in $\mathbb{D}$, $0<\lambda<\frac{1}{2}$ and
 \[
 B = \Big\lbrace \alpha\in\mathbb{D}: \vert f^{\prime}(\alpha)\vert^{2} < \varepsilon^{3} B_{\lambda}f(\alpha) \Big\rbrace.
 \]
 There is $C>0$ depending only on $\eta$ such that
 \[
 \iint \limits_{B\cap \Gamma_{\beta}(\zeta)} \vert f^{\prime}(z)\vert^{2} dA(z) \leq \varepsilon C \iint \limits_{\Gamma_{\beta^{\prime}}(\zeta)} \vert f^{\prime}(z)\vert^{2} dA(z)
 \]
 \end{lemma}
 \begin{proof}
 We write
 \[
 \iint \limits_{B\cap \Gamma_{\beta}(\zeta)} \vert f^{\prime}(\alpha)\vert^{2} dA(\alpha) = \iint \limits_{B\cap \Gamma_{\beta}(\zeta)\cap A} \vert f^{\prime}(\alpha)\vert^{2} dA(\alpha) + \iint \limits_{(B\cap \Gamma_{\beta}(\zeta))\setminus A} \vert f^{\prime}(\alpha)\vert^{2} dA(\alpha),
 \]
 where $A$ is as in lemma \ref{lemma_2}.
 The first integral is estimated by lemma \ref{lemma_2}, so it remains to show the desired result for the second integral. Integrating the relation
 \[
 \vert f^{\prime}(\alpha)\vert^{2} < \varepsilon^{3} B_{\lambda}f(\alpha) = \varepsilon^{3} \frac{1}{A(E_{\lambda}(\alpha))} \iint \limits_{E_{\lambda}(\alpha)} \vert f^{\prime}(z)\vert^{2} dA(z)
 \]
 over the set $(B\cap \Gamma_{\beta}(\zeta))\setminus A$ and using Fubini's theorem on the right side,  we get
 \begin{align}\label{epsilon_lambda}
 \iint \limits_{(B\cap \Gamma_{\beta}(\zeta))\setminus A}  \vert f^{\prime}(\alpha)\vert^{2}  dA(\alpha) &  \leq \varepsilon^{3} \iint \limits_{\Gamma_{\beta^{\prime}}(\zeta)} \vert f^{\prime}(z)\vert^{2} \Big[\iint \limits_{(B\cap \Gamma_{\beta}(\zeta))\setminus A} \frac{1}{A(E_{\lambda}(\alpha))} \chi_{E_{\lambda}(\alpha)}(z) dA(\alpha)\Big] dA(z)\nonumber\\
 & \leq \varepsilon^{3} \iint \limits_{\Gamma_{\beta^{\prime}}(\zeta)} \vert f^{\prime}(z)\vert^{2} \Big[\iint \limits_{(B\cap \Gamma_{\beta}(\zeta))\setminus A} \frac{1}{A(E_{\lambda}(\alpha))} \chi_{\Delta_{\eta}(\alpha)}(z) dA(\alpha)\Big] dA(z)
 \end{align}
 where the last inequality is justified by $E_{\lambda}(\alpha) \subseteq \Delta_{\eta}(\alpha)$. Let  $\alpha \not\in A$, i.e.
 \begin{equation}\label{notA}
 \vert f^{\prime}(\alpha)\vert^{2} \geq \frac{\varepsilon}{A(\Delta_{\eta}(\alpha))} \iint \limits_{\Delta_{\eta}(\alpha)} \vert f^{\prime}(z)\vert^{2} dA(z).
 \end{equation}
 Set $r=\eta(1 - \vert \alpha \vert)$ and suppose $\lambda<\frac{1}{2}$ and $\vert z-\alpha \vert < \frac{r}{4}$. We have that
 \begin{align}\label{Cauchy}
 \vert f^{\prime}(z)^{2} - f^{\prime}(\alpha)^{2}  \vert & = \frac{1}{2\pi} \Bigg\vert \int \limits_{\vert w-\alpha\vert = \frac{r}{2}} f^{\prime}(w)^{2} \Bigg(\frac{1}{w-z} - \frac{1}{w-\alpha}  \Bigg) dw \Bigg\vert \nonumber\\
 & = \frac{1}{2\pi} \Bigg\vert\int \limits_{\vert w-\alpha\vert = \frac{r}{2}}  f^{\prime}(w)^{2} \frac{z-\alpha}{ (w - z) (w - \alpha)}  dw\Bigg\vert.
\end{align}
For $\vert w - \alpha\vert = \frac{r}{2}$, by the subharmonicity of $\vert f^{\prime}\vert^{2}$ we have
\begin{align*}
\vert f^{\prime}(w)\vert^{2} < \frac{1}{\frac{r^{2}}{4}} \iint \limits_{\vert u-w\vert\leq\frac{r}{2}} \vert f^{\prime}(u)\vert^{2} dA(u)\leq \frac{C}{A(\Delta_{\eta}(\alpha))} \iint \limits_{\Delta_{\eta}(\alpha)} \vert f^{\prime}(u)\vert^{2} dA(u).
\end{align*}
Since $\vert w - z\vert > \frac{r}{4}$ when  $\vert w - \alpha\vert = \frac{r}{2}$, from \eqref{Cauchy} we get
\[
\vert f^{\prime}(z)^{2} - f^{\prime}(\alpha)^{2}  \vert \leq \frac{C \vert z-\alpha\vert}{r} \frac{1}{A(\Delta_{\eta}(\alpha))} \iint \limits_{\Delta_{\eta}(\alpha)} \vert f^{\prime}(u)\vert^{2} dA(u).
\]
Since we may assume that $C>2$, taking $\vert z-\alpha\vert < \frac{\varepsilon r}{2C}$, then we have $\vert z-\alpha\vert < \frac{r}{4}$ and we get
\begin{equation}\label{last_eq_lem}
\vert f^{\prime}(z)^{2} - f^{\prime}(\alpha)^{2}  \vert \leq \frac{\varepsilon}{2A(\Delta_{\eta}(\alpha))} \iint \limits_{\Delta_{\eta}(\alpha)} \vert f^{\prime}(u)\vert^{2} dA(u).
\end{equation}
Combining \eqref{notA} and \eqref{last_eq_lem}, we get
\[
\vert f^{\prime}(z)\vert^{2} >  \frac{1}{2} \vert f^{\prime}(\alpha)\vert^{2} > \lambda \vert f^{\prime}(\alpha)\vert^{2}.
\]
This means that if $\Delta^{\prime}=\lbrace z\in\mathbb{D}: \vert z-\alpha\vert < \frac{\varepsilon r}{2C}\rbrace$ then $\Delta^{\prime} \subset E_{\lambda}(\alpha)$ and 
\[
A(E_{\lambda}(\alpha))\geq A(\Delta^{\prime})= \frac{\varepsilon^{2}}{4C^{2}}r^{2}
= \frac{\varepsilon^{2}}{4C^{2}} A(\Delta_{\eta}(\alpha)).
\] 
We finally use this last inequality in \eqref{epsilon_lambda} and we complete the proof.
 \end{proof}
 
 \begin{proof}[Proof of theorem \ref{integral_theorem}.] $(ii) \Leftrightarrow (iii)$ This is easy and it is proved in \cite{Luecking81}.\\
$(iii) \Rightarrow (i)$ Let $\alpha\in \mathbb{D}\setminus B$, where $B$ is as in lemma \ref{lemma_3}, where $0<\varepsilon<1$, $0<\lambda<\frac{1}{2}$. Then $\frac{B_{\lambda}f(\alpha)}{\vert f^{\prime}(\alpha)\vert^{2}} \leq \frac{1}{\varepsilon^{3}}$ and, if we choose $\lambda<\varepsilon^{\frac{6}{\delta}}$, then, from lemma \ref{Luecking_lemma1}, we get that
 \begin{align}\label{lambda_rel}
 \frac{A(E_{\lambda}(\alpha))}{A(\Delta_{\eta}(\alpha))} > \frac{\frac{2}{\delta}\log\frac{1}{\varepsilon^{3}}}{\log\frac{1}{\varepsilon^{3}}+\frac{2}{\delta}\log\frac{1}{\varepsilon^{3}}}> 1 - \frac{\delta}{2}.
 \end{align}
 Combining \eqref{second_part2} and \eqref{lambda_rel}, we get
 \begin{align*}
A(G_{c} \cap E_{\lambda}(\alpha)) &=  A(G_{c} \cap \Delta_{\eta}(\alpha)) - A(G_{c} \cap (\Delta_{\eta}(\alpha)\setminus E_{\lambda}(\alpha)))\\
& \geq \delta A(\Delta_{\eta}(\alpha)) - A(\Delta_{\eta}(\alpha)\setminus E_{\lambda}(\alpha))\\
& = \delta A(\Delta_{\eta}(\alpha)) - A(\Delta_{\eta}(\alpha)) + A(E_{\lambda}(\alpha))\\
& \geq \delta A(\Delta_{\eta}(\alpha)) - A(\Delta_{\eta}(\alpha)) + A(\Delta_{\eta}(\alpha)) - \frac{\delta}{2} A(\Delta_{\eta}(\alpha))\\
& = \frac{\delta}{2} A(\Delta_{\eta}(\alpha))
 \end{align*}
Now let $f\in H_{0}^{p}, \zeta\in\mathbb{T}$ and $\alpha\in \Gamma_{\beta}(\zeta)\setminus B$. Then, using the last relation and $E_{\lambda}(\alpha) \subset \Delta_{\eta}(\alpha)\subset\Gamma_{\beta^{\prime}}(\zeta)$, we get
 \begin{align*}
 \frac{1}{A(\Delta_{\eta}(\alpha))} & \iint \limits_{G_{c}\cap\Gamma_{\beta^{\prime}}(\zeta)} \chi_{\Delta_{\eta}(\alpha)}(z)  \vert f^{\prime}(z) \vert^{2} dA(z) \\
 & \geq  \frac{\delta}{2 A(G_{c} \cap E_{\lambda}(\alpha))} \iint \limits_{G_{c}\cap E_{\lambda}(\alpha)}  \chi_{\Delta_{\eta}(\alpha)}(z) \vert f^{\prime}(z) \vert^{2} dA(z)\\
 & =  \frac{\delta}{2 A(G_{c} \cap E_{\lambda}(\alpha))} \iint \limits_{G_{c}\cap E_{\lambda}(\alpha)}   \vert f^{\prime}(z) \vert^{2} dA(z)
 \geq  \frac{\delta\lambda}{2}  \vert f^{\prime}(\alpha) \vert^{2}.
 \end{align*}
 Integrating the last relation over the set $\Gamma_{\beta}(\zeta)\setminus B$ and using Fubini's theorem on the left side, we have
 \begin{equation*}
 \iint \limits_{G_{c}\cap\Gamma_{\beta^{\prime}}(\zeta)} \vert f^{\prime}(z) \vert^{2} \Big[\hspace{1mm} \iint \limits_{\Gamma_{\beta}(\zeta)\setminus B} \frac{\chi_{\Delta_{\eta}(\alpha)}(z)}{A(\Delta_{\eta}(\alpha))}    dA(\alpha)\Big] dA(z) \geq  \frac{\delta\lambda}{2} \iint \limits_{\Gamma_{\beta}(\zeta)\setminus B} \vert f^{\prime}(\alpha) \vert^{2} dA(\alpha).
\end{equation*}
 With similar arguments as in relation \eqref{brackets_int}, we can show that the integral in the brackets is bounded above from a constant $C>0$ depending only on $\eta$. So, we have that
 \begin{align*}
 \iint \limits_{G_{c}\cap\Gamma_{\beta^{\prime}}(\zeta)} \vert f^{\prime}(z) \vert^{2} dA(z) & \geq  \frac{C\delta\lambda}{2} \iint \limits_{\Gamma_{\beta}(\zeta)\setminus B} \vert f^{\prime}(\alpha) \vert^{2} dA(\alpha)\\
 & =  \frac{C\delta\lambda}{2} \iint \limits_{\Gamma_{\beta}(\zeta)} \vert f^{\prime}(\alpha) \vert^{2} dA(\alpha) - \frac{C\delta\lambda}{2} \iint \limits_{\Gamma_{\beta}(\zeta)\cap B} \vert f^{\prime}(\alpha) \vert^{2} dA(\alpha).
\end{align*}
 Because of lemma \ref{lemma_3} we have that
 \[
 \iint \limits_{G_{c}\cap\Gamma_{\beta^{\prime}}(\zeta)} \vert f^{\prime}(z) \vert^{2} dA(z) \geq  \frac{C\delta\lambda}{2} \iint \limits_{\Gamma_{\beta}(\zeta)} \vert f^{\prime}(\alpha) \vert^{2} dA(\alpha) - \varepsilon \frac{C^{\prime}\delta\lambda}{2} \iint \limits_{\Gamma_{\beta^{\prime}}(\zeta)} \vert f^{\prime}(\alpha) \vert^{2} dA(\alpha)
 \]
 and so
 \[
 \iint \limits_{G_{c}\cap\Gamma_{\beta^{\prime}}(\zeta)} \vert f^{\prime}(z) \vert^{2} dA(z) + \varepsilon \frac{C^{\prime}\delta\lambda}{2} \iint \limits_{\Gamma_{\beta^{\prime}}(\zeta)} \vert f^{\prime}(\alpha) \vert^{2} dA(\alpha) \geq  \frac{C\delta\lambda}{2} \iint \limits_{\Gamma_{\beta}(\zeta)} \vert f^{\prime}(\alpha) \vert^{2} dA(\alpha).
 \]
 Hence,
 \begin{align*}
 \Big(\iint \limits_{G_{c}\cap\Gamma_{\beta^{\prime}}(\zeta)} \vert f^{\prime}(z) \vert^{2} dA(z)\Big)^{\frac{1}{2}} & + \Big(\frac{C^{\prime}\varepsilon\delta\lambda}{2}\Big)^{\frac{1}{2}} \Big(\iint \limits_{\Gamma_{\beta^{\prime}}(\zeta)} \vert f^{\prime}(\alpha) \vert^{2} dA(\alpha)\Big)^{\frac{1}{2}} \\
 & \geq  \Big(\frac{C\delta\lambda}{2}\Big)^{\frac{1}{2}} \Big(\iint \limits_{\Gamma_{\beta}(\zeta)} \vert f^{\prime}(\alpha) \vert^{2} dA(\alpha)\Big)^{\frac{1}{2}}.\\
 \end{align*}
 Applying Minkowski's inequality, we get
 \begin{align*}
\Big[\int \limits_{\mathbb{T}} \Big(\iint \limits_{G_{c}\cap\Gamma_{\beta^{\prime}}(\zeta)} \vert f^{\prime}(z) \vert^{2} dA(z)\Big)^{\frac{p}{2}}dm(\zeta)\Big]^{\frac{1}{p}} & + \Big( \frac{C^{\prime}\varepsilon\delta\lambda}{2} \Big)^{\frac{1}{2}}  \Big[\int \limits_{\mathbb{T}} \Big(  \iint \limits_{\Gamma_{\beta^{\prime}}(\zeta)} \vert f^{\prime}(\alpha) \vert^{2} dA(\alpha)\Big)^{\frac{p}{2}} dm(\zeta)\Big]^{\frac{1}{p}} \\
 & \geq \Big(\frac{C\delta\lambda}{2}\Big)^{\frac{1}{2}} \Big[ \int \limits_{\mathbb{T}}  \Big(\iint \limits_{\Gamma_{\beta}(\zeta)} \vert f^{\prime}(\alpha) \vert^{2} dA(\alpha)\Big)^{\frac{p}{2}} dm(\zeta)\Big]^{\frac{1}{p}}\\
 \end{align*}
 and so
 \begin{align}\label{almost_last}
\Big[\int \limits_{\mathbb{T}} \Big(\iint \limits_{G_{c}\cap\Gamma_{\beta^{\prime}}(\zeta)} \vert f^{\prime}(z) \vert^{2}  dA(z)\Big)^{\frac{p}{2}}dm(\zeta) & \Big]^{\frac{1}{p}}   \geq \Big(\frac{C\delta\lambda}{2}\Big)^{\frac{1}{2}} \Big[ \int \limits_{\mathbb{T}}  \Big(\iint \limits_{\Gamma_{\beta}(\zeta)} \vert f^{\prime}(\alpha) \vert^{2} dA(\alpha)\Big)^{\frac{p}{2}} dm(\zeta)\Big]^{\frac{1}{p}}\nonumber \\  
&- \Big( \frac{C^{\prime}\varepsilon\delta\lambda}{2} \Big)^{\frac{1}{2}}  \Big[\int \limits_{\mathbb{T}} \Big(  \iint \limits_{\Gamma_{\beta^{\prime}}(\zeta)} \vert f^{\prime}(\alpha) \vert^{2} dA(\alpha)\Big)^{\frac{p}{2}} dm(\zeta)\Big]^{\frac{1}{p}}.\nonumber\\
 \end{align}
 According to \eqref{Stolz_norm}, both integrals at the right side of \eqref{almost_last}, represent equivalent norms in $H_{0}^{p}$. Due to the relation between $\beta$ and $\beta^{\prime}$ there is   $C^{\prime\prime}>0$ which depends only on $\eta$, such that 
 \begin{equation}\label{equiv_stolz_norms}
 \Big[\int \limits_{\mathbb{T}} \Big(  \iint \limits_{\Gamma_{\beta^{\prime}}(\zeta)} \vert f^{\prime}(\alpha) \vert^{2} dA(\alpha)\Big)^{\frac{p}{2}} dm(\zeta)\Big]^{\frac{1}{p}}
 \leq C^{\prime\prime} \Big[ \int \limits_{\mathbb{T}}  \Big(\iint \limits_{\Gamma_{\beta}(\zeta)} \vert f^{\prime}(\alpha) \vert^{2} dA(\alpha)\Big)^{\frac{p}{2}} dm(\zeta)\Big]^{\frac{1}{p}}.
 \end{equation}
 Combining relations \eqref{almost_last} and \eqref{equiv_stolz_norms}, we get
 \begin{align*}
 \Big[\int \limits_{\mathbb{T}} \Big[\Big(\iint \limits_{G_{c}\cap\Gamma_{\beta^{\prime}}(\zeta)}  \vert f^{\prime}(z) \vert^{2} dA(z)\Big)^{\frac{p}{2}}   & dm(\zeta)  \Big]^{\frac{1}{p}} \geq \Big(\frac{C\delta\lambda}{2}\Big)^{\frac{1}{2}} \Big[ \int \limits_{\mathbb{T}}  \Big(\iint \limits_{\Gamma_{\beta}(\zeta)} \vert f^{\prime}(\alpha) \vert^{2} dA(\alpha)\Big)^{\frac{p}{2}} dm(\zeta)\Big]^{\frac{1}{p}}\\  
&- \Big( \frac{C^{\prime}\varepsilon\delta\lambda}{2} \Big)^{\frac{1}{2}}  C^{\prime\prime} \Big[ \int \limits_{\mathbb{T}}  \Big(\iint \limits_{\Gamma_{\beta}(\zeta)} \vert f^{\prime}(\alpha) \vert^{2} dA(\alpha)\Big)^{\frac{p}{2}} dm(\zeta)\Big]^{\frac{1}{p}}\\
 & = \Big(\frac{\delta\lambda}{2}\Big)^{\frac{1}{2}} [C^{\frac{1}{2}}  - \varepsilon^{\frac{1}{2}} {C^{\prime}}^{\frac{1}{2}}C^{\prime\prime}] \Vert f  \Vert_{H_{0}^{p}}.
\end{align*}
Choosing $\varepsilon$ small enough so that $C  - \varepsilon^{\frac{1}{2}} {C^{\prime}}^{\frac{1}{2}}C^{\prime\prime}>0$, we have that
 \begin{equation*}
 \Big[\int \limits_{\mathbb{T}} \Big(\iint \limits_{G_{c}\cap\Gamma_{\beta^{\prime}}(\zeta)} \vert f^{\prime}(z) \vert^{2} dA(z)\Big)^{\frac{p}{2}}dm(\zeta)\Big]^{\frac{1}{p}} \geq C \Vert f  \Vert_{H_{0}^{p}},
 \end{equation*}
 and since $G_{c} = \lbrace z\in\mathbb{D}: \vert g(z)\vert>c   \rbrace$, we have
 \begin{align*}
 \Vert S_{g}f \Vert_{H_{0}^{p}} &  \asymp \Big[\int \limits_{\mathbb{T}} \Big(\iint \limits_{\Gamma_{\beta^{\prime}}(\zeta)} \vert (S_{g}f(z))^{\prime} \vert^{2} dA(z)\Big)^{\frac{p}{2}}dm(\zeta)\Big]^{\frac{1}{p}}\\
 & =  \Big[\int \limits_{\mathbb{T}} \Big(\iint \limits_{\Gamma_{\beta^{\prime}}(\zeta)} \vert f^{\prime}(z) \vert^{2} \vert g(z) \vert^{2} dA(z)\Big)^{\frac{p}{2}}dm(\zeta)\Big]^{\frac{1}{p}} \\
 &\geq c \Big[\int \limits_{\mathbb{T}} \Big(\iint \limits_{G_{c}\cap\Gamma_{\beta^{\prime}}(\zeta)} \vert f^{\prime}(z) \vert^{2} dA(z)\Big)^{\frac{p}{2}}dm(\zeta)\Big]^{\frac{1}{p}}\geq C \Vert f  \Vert_{H_{0}^{p}}.
\end{align*}
So the integral operator $S_{g}$ has closed range.

$(i) \Rightarrow (ii)$ Let $\alpha\in\mathbb{D}$, $\zeta\in\mathbb{T}$, $\eta\in (0,1)$, $E(z_{0};r)=\lbrace z\in\mathbb{D}: \vert z-z_{0} \vert <r\rbrace $, $C(z_{0},r)=\lbrace z\in\mathbb{D}: \vert z-z_{0} \vert =r \rbrace$ and the arc $I_{\alpha}=\lbrace \zeta\in\mathbb{T}:\Gamma_{\frac{1}{2}}(\zeta)\cap D_{\eta}(\alpha)\neq \emptyset\rbrace$. It's easy to see that $\zeta\in I_{\alpha}$ is equivalent to  $\alpha\in\Gamma_{\eta^{\prime}}(\zeta)$, where $\eta^{\prime}$ depends only on $\eta$. In fact, an elementary geometric argument shows that $1-\eta^{\prime} \asymp 1-\eta$, where the underlying constants are absolute.

Set $R_{0}=\frac{1+\eta^{\prime}}{2}$. We continue with the proof by considering two cases for $\alpha$: {\bf (a)} $R_{0} \leq \vert\alpha\vert<1$ and {\bf (b)} $0\leq \vert\alpha\vert\leq R_{0}$.

{\bf Case (a) $R_{0} \leq \vert\alpha\vert \leq 1$:} Then another simple geometric argument gives $m(I_{\alpha}) \asymp \frac{1-\vert\alpha\vert}{(1-\eta^{\prime})^{\frac{1}{2}}}$ and hence:
\begin{equation}\label{I_alpha_estimation}
m(I_{\alpha}) \asymp \frac{1-\vert\alpha\vert}{(1-\eta)^{\frac{1}{2}}}.
\end{equation}

If $S_{g}$ has closed range on $H_{0}^{p}$ then there exists $C>0$ such that for every $f \in H_{0}^{p}$ we have
\begin{equation}\label{Hardy_closed_range}
C\Vert  S_{g}f \Vert_{H_{0}^{p}}^{p}  \geq \Vert  f  \Vert_{H_{0}^{p}}^{p}.
\end{equation}
Let 
\[
\psi_{\alpha}(z)=\frac{\alpha-z}{1-\overline{\alpha}z}.
\]
Then, after some calculations, we get that $\Vert  \psi_{\alpha}-\alpha  \Vert_{H^{p}}^{p} \asymp (1-\vert\alpha\vert)$.    

Setting $f=\psi_{\alpha}-\alpha$ in \eqref{Hardy_closed_range} and using $(x+y)^{p}\leq 2^{p-1}(x^{p} + y^{p})$, we get
\begin{align}\label{basic_closed_range}
1-\vert\alpha\vert & \leq C\Vert  S_{g}(\psi_{\alpha}-\alpha) \Vert_{H_{0}^{p}}^{p}
 = C\int \limits_{\mathbb{T}} \Big(\iint \limits_{\Gamma_{\frac{1}{2}}(\zeta)} \vert \psi_{\alpha}^{\prime}(z) \vert^{2} \vert g(z) \vert^{2} dA(z)\Big)^{\frac{p}{2}} dm(\zeta)\nonumber\\
 & \leq C \int \limits_{I_{\alpha}} \Big(\iint \limits_{\Gamma_{\frac{1}{2}}(\zeta)\cap G_{c} \cap D_{\eta}(\alpha)} \vert \psi_{\alpha}^{\prime}(z) \vert^{2} \vert g(z) \vert^{2} dA(z)\Big)^{\frac{p}{2}} dm(\zeta) \nonumber\\
 & + C \int \limits_{I_{\alpha}} \Big(\iint \limits_{\Gamma_{\frac{1}{2}}(\zeta)\cap (D_{\eta}(\alpha) \setminus  G_{c})} \vert \psi_{\alpha}^{\prime}(z) \vert^{2} \vert g(z) \vert^{2} dA(z)\Big)^{\frac{p}{2}} dm(\zeta) \nonumber\\
 &+ C\int \limits_{I_{\alpha}} \Big(\iint \limits_{\Gamma_{\frac{1}{2}}(\zeta) \setminus D_{\eta}(\alpha)} \vert \psi_{\alpha}^{\prime}(z) \vert^{2} \vert g(z) \vert^{2} dA(z)\Big)^{\frac{p}{2}} dm(\zeta) \nonumber\\
 & + C \int \limits_{\mathbb{T} \setminus I_{\alpha}} \Big(\iint \limits_{\Gamma_{\frac{1}{2}}(\zeta)} \vert \psi_{\alpha}^{\prime}(z) \vert^{2} \vert g(z) \vert^{2} dA(z)\Big)^{\frac{p}{2}} dm(\zeta) \nonumber\\
 & = C(I_{1} + I_{2} + I_{3} + I_{4}).
\end{align}
Using $A(D_{\eta}(\alpha)) = \frac{(1-\vert \alpha \vert^{2})^{2}}{(1-\eta^{2}\vert \alpha \vert^{2})^{2}}\eta^{2}\leq \frac{(1-\vert \alpha \vert^{2})^{2}}{(1-\eta^{2})^{2}}$, we get
\begin{align*}
I_{1} & \leq \Vert g \Vert_{\infty}^{p} \int \limits_{I_{\alpha}} \Big(\iint \limits_{G_{c} \cap D_{\eta(\alpha)}} \frac{(1-\vert \alpha \vert^{2})^{2}}{\vert 1-\overline{\alpha}z\vert^{4}}  dA(z)\Big)^{\frac{p}{2}} dm(\zeta)\\
& \leq \Vert g \Vert_{\infty}^{p} m(I_{\alpha}) \Big(\frac{A(G_{c}\cap D_{\eta}(\alpha))}{(1-\vert \alpha \vert^{2})^{2}}\Big)^{\frac{p}{2}}\\
& \leq \Vert g \Vert_{\infty}^{p} m(I_{\alpha}) \frac{1}{(1-\eta^{2})^{p}} \Big(\frac{A(G_{c}\cap D_{\eta}(\alpha))}{A(D_{\eta}(\alpha))}\Big)^{\frac{p}{2}}.
\end{align*}
Using $\vert g(z) \vert \leq c$ in $\mathbb{D} \setminus  G_{c}$ and making the change of variables $w=\psi_{\alpha}(z)$, we get
\begin{align*}
I_{2} & \leq c^{p} \int \limits_{I_{\alpha}} \Big(\iint \limits_{\mathbb{D}} \vert \psi_{\alpha}^{\prime}(z) \vert^{2}  dA(z)\Big)^{\frac{p}{2}} dm(\zeta) = c^{p} \int \limits_{I_{\alpha}} \Big(\iint \limits_{\mathbb{D}} dA(w)\Big)^{\frac{p}{2}} dm(\zeta) = c^{p} m(I_{\alpha}).
\end{align*}
We increase $I_{3}$ by extending it over $\mathbb{D} \setminus D_{\eta}(\alpha)$ and then we make the change of variables $w=\psi_{\alpha}(z)$ to get
\begin{align*}
I_{3} & \leq \Vert g \Vert_{\infty}^{p} \int \limits_{I_{\alpha}} \Big(\iint \limits_{\mathbb{D} \setminus D_{\eta}(0)}   dA(w)\Big)^{\frac{p}{2}} dm(\zeta) = \Vert g \Vert_{\infty}^{p} m(I_{\alpha}) (1-\eta^{2})^{\frac{p}{2}}.
\end{align*}
In order to estimate $I_{4}$ we have first to estimate $\iint \limits_{\Gamma_{\frac{1}{2}}(\zeta)} \vert \psi_{\alpha}^{\prime}(z) \vert^{2} dA(z)$, when $\zeta\in\mathbb{T} \setminus I_{\alpha}$. Without loss of generality we may assume that $\alpha\in[R_{0},1)$. For $j\in\mathbb{N}, j\geq 2$, we define $r_{j}=1-\frac{1}{2^{j}}$ and  consider the sets $\Omega_{1} = E(0;\frac{1}{2})$ and $\Omega_{j}= (E(0;r_{j})\setminus E(0;r_{j-1}))\cap \Gamma_{\frac{1}{2}}(\zeta)$. Then we have that $\Gamma_{\frac{1}{2}}(\zeta) = \bigcup \limits_{j=1}^{+\infty}\Omega_{j}$ and $A(\Omega_{j})\asymp \frac{1}{4^{j}}$, when $j\geq 1$. We fix $z_{j}\in\Omega_{j}$ such that $Arg(z_{j}) = Arg(\zeta)$. Then, if $z\in\Omega_{j}$, we have $\vert 1-\alpha z\vert \asymp \vert \frac{1}{\alpha}-z\vert \asymp \vert \frac{1}{\alpha}-z_{j}\vert$. Also we have that $\vert \frac{1}{\alpha}-z_{j}\vert \asymp \big\vert \frac{1}{2^{j}} + 1-\vert\alpha\vert + \vert Arg(\zeta)\vert \big\vert$. In all these relations, the underlying constants are absolute. If $\zeta\in\mathbb{T}\setminus I_{\alpha}$, then $a\not\in \Gamma_{\frac{1}{2}}(\zeta)$ which means that $1-\vert \alpha \vert < \vert Arg(\zeta)\vert$, so we have that $\vert \frac{1}{\alpha}-z_{j}\vert \asymp \big\vert \frac{1}{2^{j}} + \vert Arg(\zeta)\vert \big\vert$.  There is some $j_{0}$ so that  $\frac{1}{2^{j_{0}}} \leq \vert Arg(\zeta)\vert \leq \frac{1}{2^{j_{0}-1}} $. For $j<j_{0}$ we have $\vert Arg(\zeta)\vert < \frac{1}{2^{j}}$ which implies that $\vert \frac{1}{\alpha}-z_{j}\vert \asymp  \frac{1}{2^{j}}$  and for $j>j_{0}$ we have $\vert Arg(\zeta)\vert > \frac{1}{2^{j}}$ which implies that $\vert \frac{1}{\alpha}-z_{j}\vert \asymp \vert Arg(\zeta)\vert$. Therefore
\begin{align*}
\iint \limits_{\Gamma_{\frac{1}{2}}(\zeta)} \vert \psi_{\alpha}^{\prime}(z) \vert^{2}& dA(z) = \iint \limits_{\Omega_{1}} \frac{(1-\vert \alpha \vert^{2})^{2}}{\vert 1-\alpha z\vert^{4}} dA(z) + \sum_{j=2}^{+\infty}\iint \limits_{\Omega_{j}} \frac{(1-\vert \alpha \vert^{2})^{2}}{\vert 1-\alpha z\vert^{4}} dA(z)\\
& \asymp (1-\vert \alpha \vert^{2})^{2} + \sum_{j=2}^{j_{0}}\iint \limits_{\Omega_{j}} \frac{(1-\vert \alpha \vert^{2})^{2}}{\vert \frac{1}{\alpha}-z_{j}\vert^{4}} dA(z) + \sum_{j=j_{0}}^{+\infty} \iint \limits_{\Omega_{j}} \frac{(1-\vert \alpha \vert^{2})^{2}}{\vert \frac{1}{\alpha}-z_{j}\vert^{4}} dA(z)\\
& \asymp (1-\vert \alpha \vert^{2})^{2} + \sum_{j=2}^{j_{0}} A(\Omega_{j}) (1-\vert \alpha \vert^{2})^{2}(2^{j})^{4} + \sum_{j=j_{0}}^{+\infty} A(\Omega_{j}) \frac{(1-\vert \alpha \vert^{2})^{2}}{\vert Arg(\zeta)\vert^{4}}\\
& \asymp (1-\vert \alpha \vert^{2})^{2} + (1-\vert \alpha \vert^{2})^{2} \sum_{j=2}^{j_{0}} \frac{1}{4^{j}} 16^{j} +  \frac{(1-\vert \alpha \vert^{2})^{2}}{\vert Arg(\zeta)\vert^{4}} \sum_{j=j_{0}}^{+\infty} \frac{1}{4^{j}}.
\end{align*} 
But $\sum_{j=2}^{j_{0}} 4^{j}\asymp 4^{j_{0}} \asymp \frac{1}{\vert Arg(\zeta)\vert^{2}}$ and $\sum_{j=j_{0}}^{+\infty} \frac{1}{4^{j}}\asymp \frac{1}{4^{j_{0}}}\asymp \vert Arg(\zeta)\vert^{2}$. Therefore
\begin{align*}
\iint \limits_{\Gamma_{\frac{1}{2}}(\zeta)} \vert \psi_{\alpha}^{\prime}(z) \vert^{2}& dA(z)  \asymp (1-\vert \alpha \vert^{2})^{2} + \frac{(1-\vert \alpha \vert^{2})^{2}}{\vert Arg(\zeta)\vert^{2}}.
\end{align*} 
Since $\alpha$ is positive, there is $\phi_{0}$ such that $\mathbb{T}\setminus I_{\alpha}=[\phi_{0}, 2\pi-\phi_{0}]$ and $\phi_{0}\asymp m(I_{\alpha})$. Therefore
\begin{align*}
I_{4} & \asymp \int_{\phi_{0}}^{\pi} (1-\vert \alpha \vert^{2})^{p} d\phi + \int_{\phi_{0}}^{\pi}  \frac{(1-\vert \alpha \vert^{2})^{p}}{\phi^{p}}d\phi\\
& \asymp (1-\vert \alpha \vert^{2})^{p} + \frac{(1-\vert \alpha \vert^{2})^{p}}{\phi_{0}^{p-1}}\\
& \asymp (1-\vert \alpha \vert^{2})^{p} + \frac{(1-\vert \alpha \vert^{2})^{p}}{m(I_{\alpha})^{p-1}}.
\end{align*}

Substituting the estimates for $I_{1},I_{2},I_{3},I_{4}$ in \eqref{basic_closed_range}, we get
\begin{align*}
1-\vert \alpha \vert &\leq C\Big[\Vert g \Vert_{\infty}^{p} m(I_{\alpha}) \frac{1}{(1-\eta^{2})^{p}} \Big(\frac{A(G_{c}\cap D_{\eta}(\alpha))}{A(D_{\eta}(\alpha))}\Big)^{\frac{p}{2}} + c^{p} m(I_{\alpha}) \\
& + \Vert g \Vert_{\infty}^{p} m(I_{\alpha}) (1-\eta^{2})^{\frac{p}{2}} + (1-\vert \alpha \vert^{2})^{p} + \frac{(1-\vert \alpha \vert^{2})^{p}}{m(I_{\alpha})^{p-1}}\Big].
\end{align*}
Using \eqref{I_alpha_estimation} we get
\begin{align*}
1-\vert \alpha \vert &\leq C\Big[\Vert g \Vert_{\infty}^{p} \frac{1-\vert\alpha\vert}{(1-\eta)^{\frac{1}{2}}} \frac{1}{(1-\eta^{2})^{p}} \Big(\frac{A(G_{c}\cap D_{\eta}(\alpha))}{A(D_{\eta}(\alpha))}\Big)^{\frac{p}{2}}\\
& + c^{p} \frac{1-\vert\alpha\vert}{(1-\eta)^{\frac{1}{2}}} 
 + \Vert g \Vert_{\infty}^{p} \frac{1-\vert\alpha\vert}{(1-\eta)^{\frac{1}{2}}} (1-\eta^{2})^{\frac{p}{2}} \\
 & + (1-\vert \alpha \vert^{2})(1-\eta^{2})^{p-1} + (1-\vert \alpha \vert^{2})(1-\eta)^{\frac{p-1}{2}}\Big].
\end{align*}
Thus
\begin{align*}
C &\leq \Vert g \Vert_{\infty}^{p} \frac{1}{(1-\eta)^{\frac{2p+1}{2}}}  \Big(\frac{A(G_{c}\cap D_{\eta}(\alpha))}{A(D_{\eta}(\alpha))}\Big)^{\frac{p}{2}} +  \frac{c^{p}}{(1-\eta)^{\frac{1}{2}}} \\
& + \Vert g \Vert_{\infty}^{p} (1-\eta)^{\frac{p-1}{2}} + (1-\eta)^{p-1} + (1-\eta)^{\frac{p-1}{2}}.
\end{align*}
Choose $\eta$ close enough to 1 so that $\Vert g \Vert_{\infty}^{p} (1-\eta)^{\frac{p-1}{2}} + (1-\eta)^{p-1} + (1-\eta)^{\frac{p-1}{2}} <\frac{C}{4}$ and then set $C_{\eta}=\frac{1}{(1-\eta)^{\frac{1}{2}}}$. We have that
\begin{align*}
\frac{3C}{4} \leq \frac{\Vert g \Vert_{\infty}^{p} }{C_{\eta}^{2p+1}}  \Big(\frac{A(G_{c}\cap D_{\eta}(\alpha))}{A(D_{\eta}(\alpha))}\Big)^{\frac{p}{2}} +  \frac{c^{p}}{C_{\eta}^{\frac{1}{2}}}.
\end{align*}
Choose $c$ small enough so that $\frac{c^{p}}{C_{\eta}^{\frac{1}{2}}} <\frac{C}{4}$. Then 
\begin{align*}
\frac{C}{2} \leq \frac{\Vert g \Vert_{\infty}^{p} }{C_{\eta}^{2p+1}}  \Big(\frac{A(G_{c}\cap D_{\eta}(\alpha))}{A(D_{\eta}(\alpha))}\Big)^{\frac{p}{2}}
\end{align*}
and finally
\begin{align*}
\Big(\frac{C C_{\eta}^{2p+1}}{2\Vert g \Vert_{\infty}^{p}}\Big)^{\frac{2}{p}} \leq \frac{A(G_{c}\cap D_{\eta}(\alpha))}{A(D_{\eta}(\alpha))}
\end{align*}
or
\begin{equation*}
A(G_{c}\cap D_{\eta}(\alpha)) \geq \delta A(D_{\eta}(\alpha)),
\end{equation*}
for every $\alpha$ with $R_{0} \leq\vert \alpha\vert<1$. 

{\bf Case (b)} $0\leq \vert\alpha\vert\leq R_{0}$: There exists $\eta_{1}$, depending only on $\eta$, such that $D_{\eta}(R_{0}) \subseteq D_{\eta_{1}}(0)$. Take $\alpha^{\prime}$ so that $\vert\alpha^{\prime}\vert=R_{0}$ and $Arg(\alpha^{\prime})=Arg(\alpha)$. Then $D_{\eta}(\alpha^{\prime}) \subseteq D_{\eta_{1}}(\alpha)$. Set $\eta_{2}=\max\lbrace\eta, \eta_{1}\rbrace$. Then from case (a) for $\alpha^{\prime}$ we have
\begin{align*}
A(G_{c}\cap D_{\eta_{2}}(\alpha)) & \geq A(G_{c}\cap D_{\eta_{1}}(\alpha)) \geq A(G_{c}\cap D_{\eta}(\alpha^{\prime}))\\
& \geq \delta A(D_{\eta}(\alpha^{\prime})) \geq C\delta A(D_{\eta_1}(\alpha)) \geq C\delta A(D_{\eta_2}(\alpha)),
\end{align*}
where the constants $C>0$ depend only on $\eta$. 

Moreover, when $R_{0} \leq\vert \alpha\vert<1$, we have
\begin{align*}
A(G_{c}\cap D_{\eta_{2}}(\alpha)) & \geq A(G_{c}\cap D_{\eta}(\alpha)) \geq \delta A(D_{\eta}(\alpha)) \geq C\delta A(D_{\eta_2}(\alpha)), 
\end{align*}
where the constant $C>0$ depends only on $\eta$. So, we have proved that there are $\eta_{2}\in(0,1)$, $c>0$ and $C>0$ such that 
\[
A(G_{c}\cap D_{\eta_{2}}(\alpha)) \geq C A(D_{\eta_2}(\alpha)),
\]
for every $\alpha\in\mathbb{D}$, which is what we had to prove.
\end{proof}
\begin{remark} 
 As we can observe, the proof of the implication $(ii) \Rightarrow (i)$ in theorem \ref{integral_theorem} is valid even in the case $p=1$.
\end{remark}
\section{Closed range integral operators on BMOA space}
Let denote as $BMOA_{0}$ the space $BMOA/\mathbb{C}$. In \cite{Anderson}, A. Anderson posed the question of finding a necessary and sufficient condition for the operator $S_{g}$ to have closed range on $BMOA_{0}$.
Next, we answer this question, proving that 
conditions (ii) and (iii) of theorem \ref{integral_theorem}, for $H_{0}^{p}$, are also necessary and sufficient for the integral operator $S_{g}$ to have closed range on $BMOA_{0}$. 

Let $z_{0}\in\mathbb{D}$. The point evaluation functional of the derivative, on  $BMOA$, induced by $z_{0}$, is defined as $\Lambda_{z_{0}}f=f^{\prime}(z_{0}), f\in BMOA$.  It is easy to check that $\Lambda_{z_{0}}$ is bounded on  $BMOA$. Therefore, using Theorem 2.2 and Corollary 2.3 in \cite{Anderson}, we conclude that
the operator $S_{g}:BMOA_{0} \rightarrow BMOA_{0}$ is bounded if and only if $g\in H^{\infty}$. So, we consider $g\in H^{\infty}$ and set again $G_{c} = \lbrace z\in\mathbb{D}: \vert g(z)\vert>c\rbrace$.

The following theorem is the main result of this section. 
\begin{theorem}\label{integral_theorem_bmoa}
Let $g\in H^{\infty}$. Then the following are equivalent:
\begin{enumerate}
\item[(i)] The operator $S_{g}:BMOA_{0}\rightarrow BMOA_{0}$ has closed range 
\item[(ii)] There exist $c>0$, $\delta > 0$ and $\eta \in (0,1)$ such that
\begin{equation}\label{theorem_rel}
A(G_{c} \cap D_{\eta}(a)) \geq \delta A(D_{\eta}(a))
\end{equation}
for all $a \in \mathbb{D}$.
\end{enumerate}
\end{theorem}
Recall that the weighted Bergman space $\mathbb{A}_{\gamma}^{p}, \gamma>-1$, is defined as the set of all analytic functions $f$ in $\mathbb{D}$ such that
\[
\iint \limits_{\mathbb{D}} \vert f(z)\vert^{p} (1-\vert z\vert^{2})^{\gamma}dA(z) < \infty.
\]
We will make use of the following theorem of D. Luecking (see \cite{Luecking81}).
\begin{theorem}\label{Lue2}
Let $p\geq 1$, $\gamma>-1$ and measurable $G\subseteq\mathbb{D}$. The following assertions are equivalent.
\begin{enumerate}
\item[(i)] There exists $C>0$ such that 
\begin{equation}\label{Lue_integrals}
\iint \limits_{G} \vert f(z) \vert^{p} (1 - \vert z \vert^{2})^{\gamma} dA(z) 
\geq C \iint \limits_{\mathbb{D}} \vert f(z) \vert^{p} (1 - \vert z \vert^{2})^{\gamma} dA(z)
\end{equation}
for every $f \in \mathbb{A}_{\gamma}^{p}$.
\item[(ii)] There exist $c>0$, $\delta > 0$ and $\eta \in (0,1)$ such that
\[
A(G \cap D_{\eta}(a)) \geq \delta A(D_{\eta}(a))
\]
for all $a \in \mathbb{D}$.
\end{enumerate}
\end{theorem}
In the proof of theorem \ref{integral_theorem_bmoa}, we will use the fact that $\log\frac{1}{\vert z \vert} \asymp 1-\vert z \vert^{2}$, when $0<\delta\leq\vert z\vert < 1$, where $\delta$ is fixed but arbitrary.
\begin{proof}[Proof of theorem \ref{integral_theorem_bmoa}.]
$(ii) \Rightarrow (i)$ If \eqref{theorem_rel} holds then, because of theorem \ref{Lue2},  \eqref{Lue_integrals} also holds for $G=G_{c}$. For $\beta\in\mathbb{D}$, $z\in\mathbb{D}$ and $f\in BMOA_{0}$, we consider the function  $h_{\beta}(z) = \frac{(1-\vert\beta\vert^{2})^{\frac{1}{2}}}{1-\overline{\beta}z} f^{\prime}(z)$. It's easy to see that if $f\in BMOA_{0}$ then $h_{\beta}\in\mathbb{A}_{1}^{2}$. 
Indeed
\begin{align*}
\Vert h_\beta \Vert_{\mathbb{A}_{1}^{2}}^{2} & = \iint \limits_{\mathbb{D}} \frac{1-\vert\beta\vert^{2}}{\vert 1-\overline{\beta}z \vert^{2}} \vert f^{\prime}(z) \vert^{2} (1-\vert z\vert^{2}) dA(z)\\
& \leq C \iint \limits_{\mathbb{D}} \frac{1-\vert\beta\vert^{2}}{\vert 1-\overline{\beta}z \vert^{2}} \vert f^{\prime}(z) \vert^{2} \log\frac{1}{\vert z\vert} dA(z)\leq \Vert f\Vert_{BMOA_{0}}^{2}<\infty.
\end{align*}
Let $\beta\in\mathbb{D}$. We have that
\begin{align*}
\Vert S_{g}f \Vert_{BMOA_{0}}^{2}& = \sup \limits_{z_{0}\in\mathbb{D}} \iint \limits_{\mathbb{D}} \frac{1-\vert z_{0}\vert^{2}}{\vert 1-\overline{z_{0}}z\vert^{2}} \vert (S_{g}f(z))^{\prime} \vert^{2} \log\frac{1}{\vert  z \vert} dA(z) \\
& = \sup \limits_{z_{0}\in\mathbb{D}} \iint \limits_{\mathbb{D}} \frac{1-\vert z_{0}\vert^{2}}{\vert 1-\overline{z_{0}}z\vert^{2}} \vert f^{\prime}(z) \vert^{2} \vert g(z) \vert^{2} \log\frac{1}{\vert  z \vert} dA(z) \\
& \geq \iint \limits_{\mathbb{D}} \frac{1-\vert \beta\vert^{2}}{\vert 1-\overline{\beta}z\vert^{2}} \vert f^{\prime}(z) \vert^{2} \vert g(z) \vert^{2} \log\frac{1}{\vert  z \vert} dA(z) \\
& \geq c^{2} \iint \limits_{G_{c}} \frac{1-\vert \beta\vert^{2}}{\vert 1-\overline{\beta}z\vert^{2}} \vert f^{\prime}(z) \vert^{2} \log\frac{1}{\vert  z \vert} dA(z) \\
& = c^{2} \iint \limits_{G_{c}} \vert h_{\beta}(z)\vert^{2} \log\frac{1}{\vert  z \vert} dA(z) \\
& \geq C \iint \limits_{G_{c}} \vert h_{\beta}(z)\vert^{2} (1 - \vert  z \vert^{2}) dA(z) \\
& \geq C \iint \limits_{\mathbb{D}} \vert h_{\beta}(z)\vert^{2} (1 - \vert  z \vert^{2}) dA(z),
\end{align*}
where the last inequality is justified by theorem \ref{Lue2}.
So 
\begin{align*}
\Vert S_{g}f \Vert_{BMOA_{0}}^{2} & \geq C \iint \limits_{\mathbb{D}} \frac{1-\vert \beta\vert^{2}}{\vert 1-\overline{\beta}z\vert^{2}} \vert f^{\prime}(z) \vert^{2} \log\frac{1}{\vert  z \vert} dA(z).
\end{align*}
Taking the supremum over $\beta\in\mathbb{D}$ in the last relation we get
\[
\Vert S_{g}f \Vert_{BMOA_{0}}^{2}  \geq C \Vert f \Vert_{BMOA_{0}}^{2}.
\]

$(i) \Rightarrow (ii)$ If $S_{g}$ has closed range then there exist $C_{1}>0$ such that for every $f \in BMOA_{0}$ we have
\[
\Vert  S_{g}f \Vert_{BMOA_{0}}^{2}  \geq C_{1} \Vert  f  \Vert_{BMOA_{0}}^{2}.
\]
For $\alpha\in\mathbb{D}$, if we set $f=\psi_{\alpha}-\alpha$ in the last inequality, just as in the case of Hardy spaces and observe that $\Vert\psi_{\alpha}-\alpha\Vert_{BMOA} \asymp 1$ and $\frac{(1-\vert\beta\vert^{2})(1-\vert z \vert^{2})}{\vert 1-\overline{\beta}z\vert^{2}}<1$, for every $z,\beta\in\mathbb{D}$, then we have
\begin{align*}
C_{1} & \leq \Vert  S_{g}(\psi_{\alpha}-\alpha) \Vert_{BMOA_{0}}^{2}\\
& = \sup \limits_{\beta\in\mathbb{D}} \iint \limits_{\mathbb{D}} \frac{1-\vert\beta\vert^{2}}{\vert 1-\overline{\beta}z\vert^{2}} \vert (S_{g}(\psi_{\alpha}-\alpha)(z))^{\prime} \vert^{2} \log\frac{1}{\vert  z \vert} dA(z)\\
& \leq C  \sup \limits_{\beta\in\mathbb{D}} \iint \limits_{\mathbb{D}} \frac{1-\vert\beta\vert^{2}}{\vert 1-\overline{\beta}z\vert^{2}} \vert \psi_{\alpha}^{\prime}(z) \vert^{2} \vert g(z) \vert^{2}  (1-\vert z \vert^{2}) dA(z)\\
&  \leq C \iint \limits_{\mathbb{D}}  \vert \psi_{\alpha}^{\prime}(z) \vert^{2} \vert g(z) \vert^{2}dA(z)\\
& \leq C\Big[\Vert g\Vert_{\infty}^{2}\iint \limits_{G_{c}\cap D_{\eta}(\alpha)}  \frac{(1-\vert\alpha \vert^{2})^{2}}{\vert 1-\overline{\alpha}z \vert^{4}} dA(z)
 + c^{2} \iint \limits_{D_{\eta}(\alpha)\setminus G_{c}}  \vert \psi_{\alpha}^{\prime}(z) \vert^{2} dA(z)\\ 
 & \hspace{75mm}+ \Vert g\Vert_{\infty}^{2} \iint \limits_{\mathbb{D}\setminus D_{\eta}(\alpha)}  \vert \psi_{\alpha}^{\prime}(z) dA(z)\Big]\\
 & \leq C\Big[\Vert g\Vert_{\infty}^{2} \iint \limits_{G_{c}\cap D_{\eta}(\alpha)} \frac{1}{(1-\vert\alpha \vert^{2})^{2}} dA(z) + c^{2} \iint \limits_{\mathbb{D}}  \vert \psi_{\alpha}^{\prime}(z) \vert^{2} dA(z)\\
 & \hspace{75mm}+ \Vert g\Vert_{\infty}^{2} \iint \limits_{\mathbb{D}\setminus D_{\eta}(\alpha)}\vert \psi_{\alpha}^{\prime}(z) \vert^{2} dA(z)\Big]\\
 & = C\Big[\Vert g\Vert_{\infty}^{2} \frac{A(G_{c}\cap D_{\eta}(\alpha))}{(1-\vert\alpha \vert^{2})^{2}}  + c^{2} \iint \limits_{\mathbb{D}} dA(w) + \Vert g\Vert_{\infty}^{2} \iint \limits_{\mathbb{D}\setminus D_{\eta}(0)} dA(w)\Big]\\
 & \leq C\Big[C^{\prime}\Vert g\Vert_{\infty}^{2} \frac{A(G_{c}\cap D_{\eta}(\alpha))}{A(D_{\eta}(\alpha))}  + c^{2} + \Vert g\Vert_{\infty}^{2}(1-\eta^{2})\Big],
\end{align*} 
where $C^{\prime}$ depends only on $\eta$ and $C$ is absolute. Therefore
\begin{align*}
  C_{1}  \leq C^{\prime}\Vert g \Vert_{\infty}^{2} \frac{A(G_{c}\cap D_{\eta}(\alpha))}{A(D_{\eta}(\alpha))} 
 + c^{2} + \Vert g \Vert_{\infty}^{2} (1-\eta^{2}).
\end{align*}

First, we choose $\eta$ close enough to 1 so that $\Vert g \Vert_{\infty}^{2} (1-\eta^{2})< \frac{C_{1}}{4}$ and $c$ small enough  so that $c< \frac{C_{1}}{4}$. So 
\[
A(G_{c}\cap D_{\eta}(\alpha))\geq \frac{C}{2\Vert g \Vert_{\infty}^{2}} A(D_{\eta}(\alpha))=\delta A(D_{\eta}(\alpha)),
\]
where $C$ depends only on $\eta$.
\end{proof}
\begin{remark} 
The $Q_{p}$ space, $0<p<\infty$, is defined as the set of all analytic functions $f$ in $\mathbb{D}$ for which
\[
\sup \limits_{\beta\in\mathbb{D}} \iint \limits_{\mathbb{D}} \frac{(1-\vert\beta\vert^{2})^{p}}{\vert 1-\overline{\beta}z\vert^{2p}} \vert f^{\prime}(z) \vert^{2} (1 -\vert  z \vert^{2})^{p} dA(z) <\infty.
\]
Let denote as $Q_{p,0}$ the space $Q_{p}/\mathbb{C}$. For $\beta,z\in\mathbb{D}$ and $f\in Q_{p,0}$, we consider the functions  $h_{\beta}(z) = \frac{(1-\vert\beta\vert^{2})^{\frac{p}{2}}}{(1-\overline{\beta}z)^{p}} f^{\prime}(z)$. It's easy to see that if $f\in Q_{p,0}$ then $h_{\beta}\in\mathbb{A}_{p}^{2}$ and  using similar arguments as in the proof of theorem \ref{integral_theorem_bmoa}, we can prove that  \eqref{theorem_rel} is also necessary and sufficient for $S_{g}$ to have closed range on $Q_{p,0}\hspace{1mm}(0<p<\infty)$.
\end{remark}
\section{Closed range integral operators on Besov spaces}
Let denote as $B_{0}^{p}$ the space $B^{p}/\mathbb{C}$.
With similar arguments as in the case of $BMOA$ space we can see that
the operator $S_{g}:B_{0}^{p} \rightarrow B_{0}^{p}$ $(1<p<\infty)$ is bounded if and only if $g\in H^{\infty}$. So, we consider $g\in H^{\infty}$ and $G_{c} = \lbrace z\in\mathbb{D}: \vert g(z)\vert>c\rbrace$. We will prove that condition \eqref{theorem_rel} is also necessary and sufficient for the operator $S_{g}$ to have closed range on $B_{0}^{p}$. For the sufficiency, we observe that, if $f\in B^{p}$ then $f^{\prime}\in\mathbb{A}_{p-2}^{p}$,  the weighted Bergman space defined in the previous section, so we can use theorem \ref{Lue2}. We have
\begin{align*}
\Vert S_{g}f \Vert_{B_{0}^{p}}^{p}& = \iint \limits_{\mathbb{D}} \vert (S_{g}f(z))^{\prime} \vert^{p} (1-\vert z \vert^{2})^{p-2} dA(z)\\
& \geq \iint \limits_{G_{c}} \vert f^{\prime}(z)\vert^{p}  \vert g(z)\vert^{p} (1-\vert z \vert^{2})^{p-2} dA(z) \\
& \geq c^{p} \iint \limits_{G_{c}} \vert f^{\prime}(z)\vert^{p} (1-\vert z \vert^{2})^{p-2} dA(z) \\
&\geq C \iint \limits_{\mathbb{D}} \vert f^{\prime}(z)\vert^{p} (1-\vert z \vert^{2})^{p-2} dA(z) \\
& = C \Vert f \Vert_{B_{0}^{p}}^{p},
\end{align*}
where the last inequality is justified by theorem \ref{Lue2}. So $S_{g}$ has closed range on $B_{0}^{p}$.

If $S_{g}$ has closed range on $B_{0}^{p}$ then there exist $C_{1}>0$ such that for every $f \in B_{0}^{p}$ we have
\[
\Vert  S_{g}f \Vert_{B_{0}^{p}}^{p}  \geq C_{1} \Vert  f  \Vert_{B_{0}^{p}}^{p}.
\]
For $\alpha\in\mathbb{D}$, if we set $f=f_{\alpha}=\frac{(1 - \vert \alpha  \vert^{2})^{\frac{2}{p}} }{\frac{2\overline{\alpha}}{p} (1 - \overline{\alpha}z)^{\frac{2}{p}}}-\frac{(1 - \vert \alpha  \vert^{2})^{\frac{2}{p}}}{\frac{2\overline{\alpha}}{p}}$ in the last inequality, just as in the case of $BMOA$, and observe that $\Vert f_{\alpha}\Vert_{B_{0}^{p}} \asymp 1$ and $\vert f_{\alpha}^{\prime}(z)\vert = \frac{(1 - \vert \alpha  \vert^{2})^{\frac{2}{p}} }{ \vert 1 - \overline{\alpha}z \vert^{{\frac{2}{p}}+1}}$, then we have
\begin{align*}
& C_{1}  \leq \Vert  S_{g}f_{\alpha} \Vert_{B_{0}^{p}}^{p} = \iint \limits_{\mathbb{D}} \vert f_{\alpha}^{\prime}(z)\vert^{p} \vert g(z) \vert^{p} (1-\vert z \vert)^{p-2} dA(z)\\
 &\leq \Vert g\Vert_{\infty}^{p}\iint \limits_{G_{c}\cap D_{\eta}(\alpha)} \frac{(1 - \vert \alpha  \vert^{2})^{2} }{ \vert 1 - \overline{\alpha}z \vert^{2+p}} (1-\vert z \vert)^{p-2}  dA(z)
 + c^{p} \iint \limits_{D_{\eta}(\alpha)\setminus G_{c}}  \vert f_{\alpha}^{\prime}(z)\vert^{p} (1-\vert z \vert)^{p-2} dA(z)\\ 
 & \hspace{60mm} + \Vert g\Vert_{\infty}^{p} \iint \limits_{\mathbb{D}\setminus D_{\eta}(\alpha)} \frac{(1 - \vert \alpha  \vert^{2})^{2} }{ \vert 1 - \overline{\alpha}z \vert^{2+p}} (1-\vert z \vert)^{p-2} dA(z)\\
&\leq \Vert g\Vert_{\infty}^{p}\iint \limits_{G_{c}\cap D_{\eta}(\alpha)} \frac{1}{(1-\vert \alpha \vert^{2})^{2}}   dA(z)
 + c^{p} \iint \limits_{\mathbb{D}}  \vert f_{\alpha}^{\prime}(z)\vert^{p} (1-\vert z \vert)^{p-2} dA(z)\\ 
 & \hspace{30mm} + \Vert g\Vert_{\infty}^{p} \iint \limits_{\mathbb{D}\setminus D_{\eta}(0)} \frac{(1 - \vert \alpha  \vert^{2})^{2} }{ \vert 1 - \overline{\alpha}\psi_{\alpha}(w) \vert^{2+p}} (1-\vert \psi_{\alpha}(w) \vert)^{p-2} \vert \psi_{\alpha}^{\prime}(w) \vert^{2} dA(w)\\
 &= \Vert g\Vert_{\infty}^{p}\iint \limits_{G_{c}\cap D_{\eta}(\alpha)} \frac{1}{(1-\vert \alpha \vert^{2})^{2}}   dA(z)
 + c^{p} \Vert f_{\alpha}\Vert_{B_{0}^{p}}^{p}  + \Vert g\Vert_{\infty}^{p} \iint \limits_{\mathbb{D}\setminus D_{\eta}(0)} \frac{(1 - \vert w  \vert^{2})^{p-2} }{ \vert 1 - \overline{\alpha}w \vert^{p-2}} dA(w)\\
&\leq  \Vert g\Vert_{\infty}^{p} \frac{A(G_{c}\cap D_{\eta}(\alpha))}{(1-\vert \alpha \vert^{2})^{2}} + c^{p} \Vert f_{\alpha}\Vert_{B_{0}^{p}}^{p} + \Vert g\Vert_{\infty}^{p} \iint \limits_{\mathbb{D}\setminus D_{\eta}(0)} dA(w)\\
& \leq C^{\prime}\Vert g\Vert_{\infty}^{p} \frac{A(G_{c}\cap D_{\eta}(\alpha))}{A(D_{\eta}(\alpha))}  + C c^{p} + \Vert g\Vert_{\infty}^{p}(1-\eta^{2}),
\end{align*}
where $C^{\prime}$ depends only on $\eta$ and $C$ is absolute.
 So we have
\begin{align*}
& C_{1}  \leq C^{\prime}\Vert g\Vert_{\infty}^{p}\frac{A(G_{c}\cap D_{\eta}(\alpha))}{A(D_{\eta}(\alpha))} + C c^{p} + \Vert g\Vert_{\infty}^{p} (1-\eta^{2}).
\end{align*}
Choosing $\eta$ close enough to 1 so that $\Vert g\Vert_{\infty}^{p} (1-\eta^{2})< \frac{C_{1}}{4}$, and $c$ small enough so that $C c^{p}< \frac{C_{1}}{4}$, we get 
\[
A(G_{c}\cap D_{\eta}(\alpha))\geq \frac{C_{1}}{2C^{\prime}\Vert g \Vert_{\infty}^{p}} A(D_{\eta}(\alpha))=\delta A(D_{\eta}(\alpha)).
\]
\section*{Acknowledgements}
Many thanks to Prof. Michael Papadimitrakis for discussions about the mathematical content of this paper. His contribution was essential in order for it to take its final form.





\end{document}